\documentclass{amsart}
\usepackage{amsmath}
\usepackage{amssymb}
\usepackage{amsfonts}
\usepackage{amsthm}
\usepackage{enumerate}
\usepackage[margin=1.5in]{geometry}
\usepackage{mathtools}
\usepackage{mathrsfs}
\usepackage{hyperref}
\usepackage[table,xcdraw]{xcolor}
\usepackage[T1]{fontenc}
\usepackage[utf8]{inputenc}
\usepackage{enumitem}
\usepackage{bbm}
\usepackage[scaled]{beramono}

\usepackage[ backend=biber,
  style=alphabetic,
  maxbibnames=99,
  maxalphanames=99
]{biblatex}
\hypersetup{
    colorlinks,
    linkcolor={red!50!black},
    citecolor={blue!50!black},
    urlcolor={blue!80!black}
}

\def\bal#1\nal{\begin{align*}#1\end{align*}}
\def\ball#1\nall{\begin{align}#1\end{align}}
\def\lbal#1\lnal{\begin{flalign*}#1\end{flalign*}}

\makeatletter
\renewcommand*\env@matrix[1][\arraystretch]{%
  \edef\arraystretch{#1}%
  \hskip -\arraycolsep
  \let\@ifnextchar\new@ifnextchar
  \array{*\c@MaxMatrixCols c}}
\makeatother

\newcommand{\longhookrightarrow}{\lhook\joinrel\longrightarrow}

\newcommand{\C}{\mathbb{C}}

\newcommand{\Z}{\mathbb{Z}}
\newcommand{\Q}{\mathbb{Q}}

\theoremstyle{definition}

\newtheorem{thm}{Theorem}[section]

\newtheorem{prop}[thm]{Proposition}
\newtheorem{lem}[thm]{Lemma}

\newtheorem{rem}[thm]{Remark}

\numberwithin{equation}{section}

\newcommand{\p}{\mathfrak{p}}

\DeclareFontFamily{U}{wncy}{}
    \DeclareFontShape{U}{wncy}{m}{n}{<->wncyr10}{}
    \DeclareSymbolFont{mcy}{U}{wncy}{m}{n}
    \DeclareMathSymbol{\Sh}{\mathord}{mcy}{"58} 

\DeclareFieldFormat{postnote}{#1}
\DeclareFieldFormat{multipostnote}{#1}

\newcommand{\F}{\mathbb{F}}

\newcommand{\Sel}{\operatorname{Sel}}

\newcommand{\one}{\mathbbm{1}}

\begin{document}

\title{Infinitely many elliptic curves over $\Q(i)$\\ of exact ranks 4 and 6 with $j$-invariant 1728}

\author{Ben Savoie}
\address{Beijing International Center for Mathematical Research,
Peking University}

\email{bensavoie163@gmail.com}

\begin{abstract}
For each $r\in\{4,6\}$, we construct an explicit one-parameter family
of elliptic curves over $\Q(i)$ containing infinitely many pairwise
nonisomorphic curves genuinely defined over $\Q(i)$ with
$j$-invariant $1728$ and rank exactly $r$. We construct explicit
$\Q(i)$-rational points to bound the ranks from below. Kai's theorem on prime values of linear patterns over number fields
provides specializations with controlled local behavior, allowing us to obtain
matching upper bounds via $[1+i]$-descent. The construction extends the strategy of the author's earlier
rank-$2$ paper by replacing a symmetric Gaussian-prime configuration
with systems of binary linear forms satisfying several complementary
square identities. In the rank-$6$ case, the support vectors attached to the three
constructed points and $(0,0)$ span the self-dual Reed--Muller code
$\mathrm{RM}(1,3)$, which also occurs as the kernel of the quadratic-residue
Laplacian governing the Selmer group.
\end{abstract}
\maketitle

\markboth{\textnormal{\footnotesize Ben Savoie}}
{\textnormal{\footnotesize Exact ranks 4 and 6 over $\Q(i)$ with
$j$-invariant 1728}}
\section{Introduction}

For an elliptic curve $E$ over a number field $K$, the Mordell--Weil
theorem shows that $E(K)$ is a finitely generated abelian group. The
rank of this group is typically difficult to determine and remains
poorly understood in families. In particular, little is known about
which ranks occur infinitely often over a fixed number field.

There has been substantial recent progress on prescribed positive ranks. In 2025, the rank-$1$ case over every number field was
proved independently by Zywina \cite{zywina2025rankone} and by Koymans
and Pagano \cite{koymanspagano2025rankone}. In the same year, Zywina
proved that infinitely many elliptic curves over $\Q$ have rank exactly
$2$ \cite{zywina2025there}\footnote{Assuming the parity conjecture,
Byeon--Jeong and Jeong obtained earlier infinitude results for elliptic
curves over $\Q$ of rank exactly $2$ \cite{byeon2016infinitely,jeong2019infinitely}.}, while the author proved that there are infinitely
many elliptic curves over $\Q(i)$ with $j$-invariant $1728$ and rank
exactly $2$ \cite{savoie2025rank2}. In 2026, Zywina proved that for
every number field $K$ and every $0\leq r\leq4$, there are infinitely
many elliptic curves over $K$ of rank exactly $r$
\cite{zywina2026smallrank}.

Throughout this paper, we work with elliptic curves over $\Q(i)$
that have $j$-invariant $1728$. Every such curve is a quartic twist of
$y^2=x^3+x$ and is $\Q(i)$-isomorphic to
\begin{align}\label{eq-quartic-twist}
E_b:\quad y^2=x^3+bx
\qquad \text{for some} \quad b\in\Q(i)^\times.
\end{align}
Moreover, $E_b$ and $E_{b'}$ are isomorphic over $\Q(i)$ if and only if
$b/b'\in(\Q(i)^{\times})^4$. We refer to $b$ as the \emph{twist coefficient}. The action of $\Z[i]$ by complex multiplication is defined over
$\Q(i)$, so $E_b(\Q(i))$ is naturally a $\Z[i]$-module. Hence,
\[
\operatorname{rk}_{\Z}E_b(\Q(i))
=
2\operatorname{rk}_{\Z[i]}E_b(\Q(i)),
\]
and every Mordell--Weil rank in this family is even.

Our main theorem realizes the next two positive ranks in this family.
In particular, to our knowledge, the rank-$6$ statement is the first
infinitude theorem over a fixed number field for a prescribed exact rank
greater than $4$.

We call an elliptic curve $E/\Q(i)$ \emph{genuinely defined over
$\Q(i)$} if it is not obtained by base change from $\Q$. For curves
with $j$-invariant $1728$, this distinction is particularly relevant.
If $E/\Q$ has $j(E)=1728$, then $E$ is isomorphic to its quadratic
twist by $-1$, and \cite[10.16]{Sil} yields
\begin{align}\label{eq-base-change-double}
\operatorname{rk}E(\Q(i))
=
2\operatorname{rk}E(\Q).
\end{align}
Thus, a rank-$4$ or rank-$6$ family over $\Q(i)$ could, in principle,
arise by base change from a rank-$2$ or rank-$3$ family over $\Q$.
However, we are not aware of any theorem producing infinitely many
elliptic curves over $\Q$ with $j$-invariant $1728$ and rank exactly
$2$ or $3$, so Theorem~\ref{thm-main} cannot presently be recovered
from known results over $\Q$ by base change. More generally, base change from proper subfields can contribute substantially to the ranks occurring infinitely often over
a number field, as observed in
\cite[Section~12]{park2019heuristic}.

\begin{thm}\label{thm-main}
For each $r\in\{4,6\}$, there are infinitely many pairwise
nonisomorphic elliptic curves genuinely defined over $\Q(i)$ with
$j$-invariant $1728$ and rank exactly $r$.

More precisely, for each $r\in\{4,6\}$ there are infinitely many
$t\in\Q(i)$ such that the curves
\[
E_{r,t}:\qquad y^2=x^3+b_r(t)x
\]
are pairwise nonisomorphic over $\Q(i)$, genuinely defined over
$\Q(i)$, and have rank exactly $r$, where
$b_4(T),b_6(T)\in\Z[i][T]$ are given by
\begin{align*}
b_4(T)
&=-T(128T-i)((64-48i)T-i),\\
b_6(T)
&=(T+1)(3T+1)((-2+2i)T+i)((-6-6i)T+i)\\
&\quad\times(2iT+i)(-6iT-2-5i)((8+10i)T+2+3i)
((24-30i)T+16-11i).
\end{align*}
\end{thm}

The difficulty in Theorem~\ref{thm-main} is to prescribe the rank
exactly, rather than merely to produce large lower bounds. Indeed,
Mestre proved that infinitely many elliptic curves over $\Q$ with
$j$-invariant $1728$ have rank at least $4$ \cite{mestre1992rang}.
After base change to $\Q(i)$, these curves have rank at least $8$ by
\eqref{eq-base-change-double}. However, Mestre's explicit family is not well suited to our approach to the
exact-rank problem: its twist coefficient has irreducible factors of
degree greater than $1$ over $\Q(i)$. By contrast, after homogenization,
the twist coefficients in our families are products of binary linear
forms, allowing us to apply prime-pattern results to the individual
factors and control the pairwise residue symbols appearing in the
descent.

There is also a geometric distinction between our constructions and
the recent prescribed-rank families of Zywina
\cite{zywina2026smallrank}. His explicit families are rational
elliptic surfaces with a rational point of order $2$, whose generic
Mordell--Weil rank is at most $4$ by
\cite{oguiso1991mordell}; see
\cite[Section~2]{zywina2026smallrank}. Our rank-$4$ family is likewise
a rational elliptic surface and attains this bound, whereas the
rank-$6$ family is an elliptic K3 surface. Thus, the rank-$6$
construction necessarily lies beyond the rational surface setting.

A common ingredient in much of the recent progress on prescribed
ranks is the combination of descent with results on prime values of
linear patterns. Kai's theorem extends the Green--Tao--Ziegler theorem
on such patterns to number fields \cite{kai2023linear}, and Zywina
uses it to choose specializations with controlled bad primes. Related
additive-combinatorial input also appears in work of Koymans and
Pagano on Hilbert's Tenth Problem and rank stability
\cite{koymans2024hilbert}. The author's earlier rank-$2$ paper instead
used Tao's theorem on constellations in the Gaussian primes
\cite{tao2006gaussian}. That construction remained inside the
quadratic-twist subfamily of the $j=1728$ quartic twists: the curves
had the form
\[
y^2=x^3-b^2x,
\]
where
\[
b
=
\prod_{j=1}^4\bigl(\beta+i^j(1+i)k\bigr)
=
\beta^4+4k^4,
\qquad
\beta\in\Z[i],\quad k\in\Z.
\]

To reach ranks $4$ and $6$, we work instead with the full family of
quartic twists and use more flexible systems of binary linear forms
designed to satisfy several complementary square identities.

The constructions begin with the elementary observation that if
\[
b=ef\qquad\text{and}\qquad e+f=R^2,
\]
then $(e,eR)\in E_b(\Q(i))$. The image of this point under the
$[1+i]$-Kummer map is represented by the squareclass of $e$. We will
refer to such a factorization $b=ef$, with $R \neq 0$, as a \emph{square-split
factorization}.\footnote{This elementary point-producing device is
standard in the construction of rational points on elliptic curves
with rational $2$-torsion. For example, for a curve
$y^2=x^3+Ax^2+Bx$, Dujella--Kazalicki--Peral
\cite[Section~1.1]{dujella2021elliptic} consider divisors $d$ of $B$
and impose the equivalent condition $d+A+B/d=\square$.}
The non-torsion point used in \cite{savoie2025rank2} arises from
a square-split factorization, with $e=4\beta^2k^2$. For the
higher-rank families constructed here, we seek $b\in\Z[i]$ admitting
several square-split factorizations $b=e_jf_j$ for which the corresponding squareclasses $[e_j]$ are independent modulo the torsion class $[b]$.

Realizing several such factorizations restricts the
number and arrangement of the linear factors dividing $b$. Suppose
that, in homogeneous variables $X$ and $Y$, we have
\[
B=u\ell_1\cdots\ell_n,
\qquad
u\in\Q(i)^\times,
\]
where the $\ell_j$ are binary linear forms. If $e_j$ and $f_j$ are
complementary subproducts satisfying
\[
e_j+f_j=R_j^2,
\]
then $e_j$ and $f_j$ must have the same degree, and this degree must
be even. Hence, $4\mid n$, so the first possibility is four factors.

Suppose then that $B=u\ell_1\ell_2\ell_3\ell_4$ admits two square-split factorizations with independent associated
squareclasses. Up to relabeling, their supports are given by the two partitions
\[
\{\ell_1,\ell_2\}\sqcup\{\ell_3,\ell_4\}
\qquad\text{and}\qquad
\{\ell_1,\ell_3\}\sqcup\{\ell_2,\ell_4\}.
\]
We therefore seek binary linear forms $\ell_1,\ldots,\ell_4$ and
$c_1,c_2\in\Q(i)^\times$ such that
\[
c_1\ell_1\ell_2+\frac{u}{c_1}\ell_3\ell_4
\qquad\text{and}\qquad
c_2\ell_1\ell_3+\frac{u}{c_2}\ell_2\ell_4
\]
are both squares. The resulting points give two Kummer classes that are independent
modulo the torsion class, and therefore give $\Z[i]$-rank at least $2$,
so ordinary rank at least $4$.

With four factors, this support pattern cannot produce a third
independent non-torsion class. Indeed, the four factors have only
three $2+2$ partitions. The corresponding support vectors have weight
$2$, while the torsion point $(0,0)$ has Kummer image $[b]$, corresponding to the all-ones vector. Together these span only the $3$-dimensional even-weight subspace of $\F_2^4$, so modulo $[b]$ there are at most two independent non-torsion Kummer classes. Since the number of factors must be divisible by $4$, the next possibility is eight factors.

With eight factors, we seek three square-split factorizations corresponding to $4+4$ partitions that give three independent non-torsion Kummer classes. Since their support vectors must lie in the kernel of the symmetric quadratic-residue Laplacian, we seek a self-dual subspace of $\F_2^8$ containing these vectors and the all-ones vector corresponding to $(0,0)$. Up to relabeling, this determines the support pattern: the eight factors may be indexed by $\F_2^3$ so that the three square-split factorizations correspond to the coordinate hyperplanes
\[
H_j=\{v\in\F_2^3:v_j=0\},
\qquad
j=1,2,3.
\]
The resulting $4$-dimensional space is the Reed--Muller code
$\mathrm{RM}(1,3)$, and for our rank-$6$ specializations it is exactly the kernel governing the Selmer group calculation.

The preceding combinatorics determine the number of factors and the
complementary products that must occur in the construction. Finding
explicit families then amounts to solving the resulting square
identities for homogeneous binary linear forms. The solutions we use
give the factors of $b_4$ and $b_6$ appearing in
Theorem~\ref{thm-main}.

The factor counts above refer to the homogeneous binary forms. In the
rank-$4$ family, one factor becomes constant after setting $Y=1$, so
$b_4(T)$ has degree $3$. Scaling $(X,Y)$ changes the twist coefficient
by a fourth power, so the $\Q(i)$-isomorphism class depends only on
$T=X/Y$. We therefore state the families in terms of $T$ and return
to homogeneous coordinates only when applying Kai's theorem and
computing residue symbols.

Kai's theorem is then used to choose infinitely many specializations
for which the linear forms take Gaussian prime values while satisfying
the local conditions required for the descent. The upper bounds on
the rank are obtained by descent through the endomorphism $[1+i]$.
The necessary local solubility conditions are special cases of the
graph-theoretic $2$-isogeny descent developed in
\cite{kling2024computing}, but we derive them here for
completeness. In rank $4$, the classes of the two constructed points force the
quadratic-residue graph to be the star $K_{1,3}$ or its complement,
and either possibility gives the sharp Selmer bound. In rank $6$, the
two possible quadratic-residue graphs have Laplacians with common
kernel $\mathrm{RM}(1,3)$, so the same structure governing the
constructed points also forces the Selmer bound to be sharp.

The paper is organized as follows. In Section~\ref{sec-descent} we
recall the $[1+i]$-descent used throughout, determine the torsion
subgroup of every quartic twist $E_b/\Q(i)$, record the consequence
of Kai's theorem needed for simultaneous Gaussian-prime values, and
give a criterion for genuine definition over $\Q(i)$. We also record
the combinatorics of square-split factorizations and their support
patterns. Sections~\ref{sec-rank4} and
\ref{sec-rank6} treat ranks $4$ and $6$, respectively.

\section{Descent and linear prime specializations}\label{sec-descent}

\subsection{\texorpdfstring{The $[1+i]$-descent}{The [1+i]-descent}}

Let $b\in\Q(i)^\times$ and let $E_b$ be as in
\eqref{eq-quartic-twist}. Set $P_0=(0,0)$. We use descent by the
degree-$2$ CM endomorphism $[1+i]$ of $E_b$. This is equivalent to
the standard $2$-isogeny descent after identifying its target with
$E_b$. Indeed, the isogeny
\begin{align}\label{eq-phi-isogeny}
\varphi:E_b&\longrightarrow E_{-4b},\notag\\[-1ex]
(x,y)&\longmapsto
\left(\frac{y^2}{x^2},\frac{y(b-x^2)}{x^2}\right)
\end{align}
has kernel $\{O,P_0\}$. Since $[i](x,y)=(-x,iy)$, the endomorphism
$[1+i]$ has the same kernel. Then, by the addition law on $E_b$, the $\Q(i)$-isomorphism
\begin{align}\label{eq-psi-isomorphism}
\psi_b:E_{-4b}&\overset{\sim}{\longrightarrow}E_b,\notag\\[-1ex]
(X,Y)&\longmapsto
\left(\frac{X}{(1+i)^2},-\frac{Y}{(1+i)^3}\right)
\end{align}
satisfies
\begin{align}\label{eq-one-plus-i-factorization}
[1+i]=\psi_b\circ\varphi.
\end{align}

For an isogeny $\alpha:E\to E'$ over $\Q(i)$, the
$\alpha$\textit{-Selmer group} of $E$ is
\[
\Sel_{\alpha}(E/\Q(i))
=
\left\{
\xi\in H^1(\Q(i),E[\alpha]):
\operatorname{res}_v(\xi)\in
\operatorname{im}(\delta_{\alpha,\Q(i)_v})
\text{ for every place }v
\right\}.
\]
Here $\delta_{\alpha,F}$ denotes the Kummer map
\[
\delta_{\alpha,F}:
E'(F)/\alpha E(F)\longrightarrow H^1(F,E[\alpha]),
\]
which is the connecting homomorphism arising from the short exact sequence
\[
0\longrightarrow E[\alpha]\longrightarrow E \overset{\alpha}{\longrightarrow} E'\longrightarrow 0.
\]

For $F=\Q(i)$ or $F=\Q(i)_v$, the common kernel
\[
E_b[\varphi]=E_b[1+i]=\{O,P_0\}
\]
and \eqref{eq-one-plus-i-factorization} identify the $\varphi$-Kummer
map with the $[1+i]$-Kummer map after applying $\psi_b$ to the target
of $\varphi$. Under the usual squareclass descriptions (see below), $\psi_b$
scales the $x$-coordinate by $(1+i)^{-2}$, which is a square in
$F^\times$, so the Kummer class is unchanged. Thus the local Kummer
images agree at every place, and there is a canonical identification
\begin{align}\label{eq-selmer-phi-equals-one-plus-i}
\Sel_{\varphi}(E_b/\Q(i))
\cong
\Sel_{1+i}(E_b/\Q(i)).
\end{align}
We work throughout with the $[1+i]$-Selmer group, using the
equivalent $\varphi$-descent only for the explicit Kummer
map and homogeneous-space calculations.

The $\operatorname{Gal}(\overline F/F)$-modules $E_b[1+i]$ and $\mu_2$ are isomorphic via $P_0 \mapsto -1$. By Kummer theory, we then have
\[
H^1(F,E_b[1+i])\cong F^\times/(F^\times)^2.
\]
Under this identification, the Kummer map is given by (see \cite[Section~X.4]{Sil})
\begin{align}\label{eq-kummer-x}
\delta_{1+i,F}(P)=
\begin{cases}
1,&P=O,\\
[b],&P=P_0,\\
[x],&P=(x,y)\text{ with }x\neq 0,
\end{cases}
\end{align}
where $[a]$ denotes the squareclass $a \cdot (F^\times)^2$. For $F=\Q(i)$ we write $\delta_b$ for the Kummer map.   

The Selmer group controls the possible Kummer classes of rational
points and provides an upper bound on the
Mordell--Weil rank. In general, such a bound involves an isogeny and
its dual, but in our setting the CM structure makes the relation
especially simple.

\begin{lem}\label{lem-rank-selmer}
For every $b\in\Q(i)^\times$,
\begin{align}\label{rank-Z[i]-dim}
\dim_{\F_2}E_b(\Q(i))/(1+i)E_b(\Q(i))
=\operatorname{rk}_{\Z[i]}E_b(\Q(i))+1.
\end{align}
Consequently, the Mordell--Weil rank of $E_b$ satisfies
\[
\operatorname{rk}_{\Z}E_b(\Q(i))
\leq
2\dim_{\F_2}\Sel_{1+i}(E_b/\Q(i))-2.
\]
\end{lem}

\begin{proof}
Since $E_b(\Q(i))$ is a $\Z[i]$-module, we have the decomposition
\[
E_b(\Q(i))
\cong
\Z[i]^r\oplus T_b, 
\]
where $r:=\operatorname{rk}_{\Z[i]}E_b(\Q(i))$ and $T_b: =E_b(\Q(i))_{\mathrm{tors}}$. The free part contributes $r$ dimensions after reduction modulo $1+i$, while
\[
\#T_b/(1+i)T_b=\#T_b[1+i]=2,
\]
since the kernel of $[1+i]$ is $\{O,P_0\}$. This proves \eqref{rank-Z[i]-dim}.  
By the definition of the Selmer group, the global connecting
homomorphism gives an injection
\begin{align}\label{kummer-injection}
E_b(\Q(i))/(1+i)E_b(\Q(i))
\longhookrightarrow
\Sel_{1+i}(E_b/\Q(i)).
\end{align}
The rank bound now follows from \eqref{rank-Z[i]-dim} and $\operatorname{rk}_{\Z}E_b(\Q(i))=2r.$ 
\end{proof}

We will use the following consequence of the $[1+i]$-descent, which combines the lower bound coming from independent Kummer classes with a matching Selmer upper bound.

\begin{lem}\label{lem-sharp-descent}
Let $b\in\Q(i)^\times$, and suppose that
$P_1,\ldots,P_m\in E_b(\Q(i))$ are such that
\[
[b],\delta_b(P_1),\ldots,\delta_b(P_m)
\]
are linearly independent in
$\Q(i)^\times/(\Q(i)^\times)^2$. If
\[
\dim_{\F_2}\Sel_{1+i}(E_b/\Q(i))\leq m+1,
\]
then
\[
\operatorname{rk}_{\Z}E_b(\Q(i))=2m.
\]
Moreover, under the identification in \eqref{kummer-injection},
\begin{align}\label{selmer-equality-in-lem}
E_b(\Q(i))/(1+i)E_b(\Q(i))
=
\Sel_{1+i}(E_b/\Q(i)),
\end{align}
and
\[
\Sh(E_b/\Q(i))[2^\infty]=0.
\]
\end{lem}

\begin{proof}
Since $[b]=\delta_b((0,0))$, the linear independence hypothesis
implies
\[
\dim_{\F_2}E_b(\Q(i))/(1+i)E_b(\Q(i))
=
\dim_{\F_2}\operatorname{im}(\delta_b)
\geq m+1.
\]
By \eqref{kummer-injection}, we then have 
\[
m+1
\leq
\dim_{\F_2}E_b(\Q(i))/(1+i)E_b(\Q(i))
\leq
\dim_{\F_2}\Sel_{1+i}(E_b/\Q(i))
\leq m+1.
\]
Thus, \eqref{selmer-equality-in-lem} holds, and
Lemma~\ref{lem-rank-selmer} yields $\operatorname{rk}_{\Z}E_b(\Q(i))=2m.$

Finally, the short exact sequence \cite[Theorem~X.4.2(a)]{Sil}
\[
0
\longrightarrow
E_b(\Q(i))/(1+i)E_b(\Q(i))
\longrightarrow
\Sel_{1+i}(E_b/\Q(i))
\longrightarrow
\Sh(E_b/\Q(i))[1+i]
\longrightarrow
0
\]
and \eqref{selmer-equality-in-lem} imply $\Sh(E_b/\Q(i))[1+i]=0.$ 
Since $2=-i(1+i)^2$, every element of
$\Sh(E_b/\Q(i))[2^\infty]$ is killed by some power of $1+i$.
If this group were nonzero, an element killed by a minimal positive
power of $1+i$ would give a nonzero element of
$\Sh(E_b/\Q(i))[1+i]$, a contradiction. Hence,
\[
\Sh(E_b/\Q(i))[2^\infty]=0.
\]
\end{proof}

\subsection{The odd local descent equations}\label{subsection-odd-local-descent}

In this subsection, we translate the local solubility conditions at the odd primes dividing $b$ into a system of linear equations over $\F_2$, yielding an upper bound for the dimension of the $[1+i]$-Selmer group in terms of a quadratic-residue Laplacian. This is a specialization of the graph-theoretic $2$-isogeny descent developed in \cite{kling2024computing}; we will provide a direct formulation adapted to the families considered here.

Every odd Gaussian prime has a unique primary associate, where
$\alpha\in\Z[i]$ is primary if
\[
\alpha\equiv1\pmod{(1+i)^3}.
\]
Let $\pi$ be an odd Gaussian prime with prime ideal
$\mathfrak p=(\pi)$. For $a\in\Q(i)^\times$, write
\[
a=\pi^{v_{\mathfrak p}(a)}u
\qquad \text{for some}\quad
u\in\Z[i]_{\mathfrak p}^\times.
\]
We let
\[
\left(\frac{a}{\pi}\right)\in\{-1,1\}
\]
denote the Gaussian Legendre symbol of the $\mathfrak p$-adic unit
part $u$. Thus
\[
\left(\frac{a}{\pi}\right)=1
\quad\Longleftrightarrow\quad
u\pmod{\mathfrak p}
\text{ is a square in }
(\Z[i]/\mathfrak p)^\times.
\]
Equivalently, since $\mathfrak p$ has odd residue characteristic,
\[
\left(\frac{a}{\pi}\right)=1
\quad\Longleftrightarrow\quad
u\in(\Q(i)_{\mathfrak p}^\times)^2.
\]
For the rest of the paper, we simply call this the Legendre symbol. 

For a primary odd element with prime factorization
$\rho=\pi_1\cdots\pi_r$, we extend the notation multiplicatively by
\[
\left(\frac{a}{\rho}\right)
=
\prod_{j=1}^r\left(\frac{a}{\pi_j}\right).
\]

For distinct primary
Gaussian primes $\pi,\rho$, quadratic reciprocity (see \cite[Proposition~5.1]{lemmermeyer2013reciprocity}) gives 
\[
\left(\frac{\pi}{\rho}\right)=\left(\frac{\rho}{\pi}\right). 
\] 
We will also use the supplementary laws (equivalent to those in \cite[Proposition~5.1]{lemmermeyer2013reciprocity})
\begin{align}\label{eq-quadratic-supplementary}
\left(\frac{i}{\pi}\right)&=(-1)^{(1-a)/2},\\
\left(\frac{1+i}{\pi}\right)&=(-1)^{(a-c-c^2-1)/4},
\end{align}
for a primary prime $\pi=a+ci$. For every primary odd Gaussian prime $\pi$, define
\[
\nu_\pi:=\log_{-1}\left(\frac{i}{\pi}\right) \in \F_2,  
\]
where $\log_{-1}:\{\pm1\}\to\F_2$ is defined by
$\log_{-1}(1)=0$ and $\log_{-1}(-1)=1$.
Equivalently,
\[
\nu_\pi=
\begin{cases}
0 &\text{if }\pi\equiv1\pmod 4,\\
1 &\text{if }\pi\equiv3+2i\pmod 4.
\end{cases}
\]

Given distinct primary primes $\pi_1,\dots,\pi_N$, define the binary
\textit{quadratic-residue Laplacian}
\[
L \in\operatorname{Mat}_{N\times N}(\F_2)
\]
by
\begin{align}\label{eq-L-definition}
L_{j,k} = \begin{cases}
    \log_{-1}\left(\frac{\pi_k}{\pi_j}\right) &\text{if }j \neq k,\\[.5em]
    \sum_{l\neq j}L_{j,l} &\text{if }j = k.
\end{cases}
\end{align}
Equivalently, $L$ is the Laplacian over $\F_2$ of the graph with vertices $\{\pi_1, \dots, \pi_N\}$, where $\pi_j$ and $\pi_k$ are adjacent precisely when  
\[
\left(\frac{\pi_k}{\pi_j}\right)=-1.
\]
We call this graph the \textit{quadratic-residue graph}. By quadratic reciprocity, $L$ is symmetric, and by construction $L\one=0$.

For a nonzero prime ideal $\mathfrak p\subset\Z[i]$, let
$v_{\mathfrak p}$ denote the normalized valuation on
$\Q(i)$, so that $v_{\mathfrak p}(\pi)=1$ whenever
$\mathfrak p=(\pi)$.

\begin{prop}[Odd local descent equations]\label{prop-selmer-linear-system}
Let
\[
b=i^{s_b}\pi_1\cdots\pi_N
\qquad \text{with}\quad 
s_b\in\{0,1,2,3\},
\]
where the $\pi_j$ are distinct primary odd Gaussian primes with corresponding prime ideals $\mathfrak p_j = (\pi_j)$, and define the vector 
\[
\nu=(\nu_{\pi_1},\dots,\nu_{\pi_N})^T \in \F_2^N.
\]
Then every class in $\Sel_{1+i}(E_b/\Q(i))$ has a square-free
representative
\[
d=i^s\prod_{j=1}^N\pi_j^{x_j},
\qquad
s\in\F_2,\quad
x=(x_j)\in\F_2^N,
\]
satisfying
\begin{align}\label{eq-selmer-squarefree-system}
\bigl(L+\overline{s_b}\operatorname{diag}(\nu)\bigr)x=s\nu,
\end{align}
where $\overline{s_b} \in \F_2$ is the reduction of $s_b$ modulo $2$. Consequently, the map
\[
[d]=\bigg[i^s\prod_{j=1}^N\pi_j^{x_j}\bigg]
\longmapsto
(s,x)\in\F_2^{N+1},
\]
identifies $\Sel_{1+i}(E_b/\Q(i))$ with a subspace of
\[
\left\{
(s,x)\in\F_2\times\F_2^N:
\bigl(L+\overline{s_b}\operatorname{diag}(\nu)\bigr)x=s\nu
\right\}. 
\]
\end{prop}

\begin{proof} By the identification
\eqref{eq-selmer-phi-equals-one-plus-i} and \cite[Section~X.4]{Sil},
every $[1+i]$-Selmer class is represented by a square-free
divisor $d$ of $2b$, modulo squares. Moreover,
\cite[Section~X.4]{Sil} associates to $d$ the homogeneous space
\begin{align}\label{eq-homogeneous-space}
C_d:\quad
dw^2=d^2-4bz^4.
\end{align}
Its weighted-projective closure in $\mathbb P(2,1,1)$ is given by
\[
dW^2=d^2T^4-4bZ^4.
\]
At a point at infinity, $T=0$, so $dW^2=-4bZ^4.$ Since $Z\neq0$ and $-4$ is a square in $\Q(i)$, a point at infinity
over a completion $F$ exists if and only if $b/d$ is a square in $F$.
 
Since $\Z[i]$ has class number $1$, a square-free representative has
the form
\[
d=i^s(1+i)^t\prod_{j=1}^N\pi_j^{x_j},
\qquad
s,t,x_j\in\F_2.
\]

Suppose that $t=1$. A point at infinity on $C_d$ would require $b/d$ to be a square in $\Q(i)_{1+i}$, which is
impossible because
\[
v_{1+i}(b/d)=-1.
\]
An affine point also leads to a contradiction, since the two terms on the right side of
\eqref{eq-homogeneous-space} have distinct even valuations
\[
2
\qquad\text{and}\qquad
4+4v_{1+i}(z),
\]
while the left side has odd valuation
\[
1+2v_{1+i}(w).
\]
Thus, no affine point exists either, so we must have $t=0$.

Now fix $j\in\{1,\dots,N\}$ and let
\[
\pi=\pi_j,\qquad
\mathfrak p=\mathfrak p_j,\qquad
\Pi=\prod_{k=1}^N\pi_k.
\]

If $x_j=0$, then $d$ is a $\mathfrak p$-adic unit and $v_{\mathfrak p}(b/d)=1,$ so there is no point at infinity over $\Q(i)_{\mathfrak p}$. Thus, any local point is a solution $(z,w)$ to the affine equation \eqref{eq-homogeneous-space}. If
$v_{\mathfrak p}(z)<0$, the second term on the right side of
\eqref{eq-homogeneous-space} has odd valuation
$1+4v_{\mathfrak p}(z)$, while the left side has even valuation.
Hence, $v_{\mathfrak p}(z)\geq 0$. This implies $d \equiv w^2 \pmod{\p}$, so that
\begin{align}\label{eq-local-squarefree-case0}
\left(\frac{d}{\pi}\right)=1.
\end{align}
By the multiplicativity of the Legendre symbol, taking $\log_{-1}$ of
\eqref{eq-local-squarefree-case0} yields
\[
\sum_{k\neq j}L_{j,k}x_k
=
s\nu_{\pi_j}.
\]

Suppose instead that $x_j=1$. Let  $d_0,\Pi_0\in\Z[i]_{\mathfrak p}^\times$ such that $d = d_0 \pi$ and $\Pi = \Pi_0 \pi$. After dividing
\eqref{eq-homogeneous-space} by $\pi$, we obtain
\begin{align}\label{eq-in-proof-of-LSC}
d_0w^2
= \pi d_0^2 - 4i^{s_b} \Pi_0 z^4.
\end{align}
We claim that
\begin{align}\label{eq-local-squarefree-case1}
\left(\frac{i^{s_b} \Pi_0/d_0}{\pi}\right)=1.
\end{align}
If $v_{\mathfrak p}(z)>0$, the right side of \eqref{eq-in-proof-of-LSC} has valuation $1$, whereas
the left side has even valuation, which is impossible. If
$v_{\mathfrak p}(z)=0$, reduction modulo $\mathfrak p$ immediately
gives \eqref{eq-local-squarefree-case1}. If
$v_{\mathfrak p}(z)<0$, dividing by $z^4$ gives
\[
d_0\left(\frac{w}{z^2}\right)^2
=
\frac{\pi d_0^2}{z^4}-4i^{s_b} \Pi_0.
\]
Here $w/z^2$ is a unit, while the first term on the right is divisible
by $\mathfrak p$, so reduction modulo $\mathfrak p$ again gives
\eqref{eq-local-squarefree-case1}. Finally, at a point at infinity, we have \eqref{eq-local-squarefree-case1} because 
\[
\frac{b}{d}=\frac{i^{s_b}\Pi_0}{d_0} \in (\Q(i)_\p^\times)^2. 
\]

Note that 
\[
\frac{i^{s_b} \Pi_0}{d_0}
=i^{s_b-s}\prod_{k\neq j}\pi_k^{1-x_k},
\]
so taking $\log_{-1}$ of \eqref{eq-local-squarefree-case1} yields
\[
(\overline{s_b}+s)\nu_{\pi_j} +\sum_{k\neq j}(1-x_k)\log_{-1}\left(\frac{\pi_k}{\pi_j}\right)=0.
\]
By the definition of $L$, this is equivalent to
\[
\sum_{k\neq j}L_{j,k}(1-x_k)
+
\overline{s_b}\nu_{\pi_j}
=
s\nu_{\pi_j}.
\]
Since $L_{j,j}=\sum_{k\neq j}L_{j,k},$ 
this is in turn equivalent to 
\[
\sum_{k\neq j}L_{j,k}x_k
+
L_{j,j}
+
\overline{s_b}\nu_{\pi_j}
=
s\nu_{\pi_j}.
\]
Thus, the two cases $x_j \in \{0,1\}$ can both be written as
\[
(Lx)_j+\overline{s_b}\nu_{\pi_j}x_j
=
s\nu_{\pi_j}.
\]
Since $j$ was arbitrary, this proves
\eqref{eq-selmer-squarefree-system}. The association of a Selmer
class with $(s,x)$ is injective, so the stated dimension bound
follows.
\end{proof}

\subsection{\texorpdfstring{Torsion in the $j=1728$ family}
{Torsion in the j=1728 family}}

We now determine the $\Q(i)$-torsion subgroup of every quartic twist
$E_b$ over $\Q(i)$.\footnote{A proof of the same statement appeared in an
earlier version of \cite{savoie2025rank2}, using division polynomials.
That argument was removed from the current version to streamline the
paper. The proof here instead uses CM and a halving criterion.}

\begin{prop}\label{prop-torsion} Let $b\in\Q(i)^\times$. Then
\[
E_b(\Q(i))_{\mathrm{tors}}\cong
\begin{cases}
\Z/2\Z\oplus\Z/4\Z,
& b\in-(\Q(i)^\times)^4,\\[1ex]
(\Z/2\Z)^2,
& b\in -(\Q(i)^\times)^2
\ \text{ and }\ b\notin-(\Q(i)^\times)^4,\\[1ex]
\Z/10\Z,
& b\in(-1\pm 2i)(\Q(i)^\times)^4,\\[1ex]
\Z/2\Z,
& \text{otherwise.}
\end{cases}
\]
\end{prop}

\begin{proof}
The point $P_0=(0,0)$ always has order $2$. By Najman's
classification of torsion over $\Q(i)$
\cite[Theorem~2(i)]{najman2010complete}, the torsion subgroup is one
of the groups in Mazur's theorem or $\Z/4\Z\oplus\Z/4\Z$. Since
$P_0$ is rational, it is therefore enough to determine the
$2$-primary torsion and then whether points of order $3$ or $5$
can occur.

We first determine the $2$-primary torsion. If $-b$ is not a square
in $\Q(i)$, then
\[
E_b(\Q(i))[2]=\{O,P_0\}.
\]
There is no rational point $Q$ of order $4$. Indeed, we would then
have $2Q=P_0$, and $R:=[i-1]Q$ would be a rational point of order $2$, since
\[
2R=[i-1]P_0=O.
\]
Moreover, $R\neq O$ since
$\ker[i-1]=\{O,P_0\}$, while $R\neq P_0$ since otherwise
$[i+1]Q=O$, contradicting
$\ker[i+1]=\{O,P_0\}$. Thus, $R$ would be a second nonzero rational
point of order $2$, contradicting
\[
E_b(\Q(i))[2]=\{O,P_0\}.
\]
Hence, the $2$-primary torsion is $\Z/2\Z$ if
$-b\notin(\Q(i)^\times)^2$.

Suppose now that $-b = a^2$ for some $a\in\Q(i)^\times$. Then
\[
E_b(\Q(i))[2]
=
\{O,P_0,(a,0),(-a,0)\}.
\]
By the halving criterion of
\cite[Theorem~2.1]{bekker2018division}, a $\Q(i)$-point $(x_0, y_0)$ lies in $2 E_b (\Q(i))$ if and only if $x_0$, $x_0+a$, and $x_0-a$ are all
squares in $\Q(i)$. It follows that $P_0 \in 2 E_b(\Q(i))$ if and only if $b \in - (\Q(i)^\times)^4.$ Moreover, neither $(a,0)$ nor $(-a,0)$ is
divisible by $2$, since the halving criterion would then imply that $2$ is
a square in $\Q(i)$. Thus, the $2$-primary torsion is
$(\Z/2\Z)^2$ if $b\notin-(\Q(i)^\times)^4$.

Now suppose that $b=-c^4$ for some $c \in \Q(i)^\times$. Then $P_0$ is divisible by $2$, so the $2$-primary torsion contains
$\Z/2\Z\oplus\Z/4\Z$. We claim that it cannot contain a point $Q$ of order $8$.
Indeed, setting $A=2Q$, the point $4Q=2A$ is a nonzero rational
$2$-torsion point divisible by $2$, and must therefore equal $P_0$.
As before,
\[
R=[i-1]A
\]
is a nonzero rational $2$-torsion point distinct from $P_0$. But
\[
R=2[i-1]Q,
\]
so $R$ is also divisible by $2$, contradicting the halving criterion.
The possibility $\Z/4\Z\oplus\Z/4\Z$ is similarly excluded, since it
would make a nonzero $2$-torsion point other than $P_0$ divisible by
$2$. Hence, the $2$-primary torsion is exactly $\Z/2\Z\oplus\Z/4\Z$ if $b \in - (\Q(i)^\times)^4.$

It remains to consider odd torsion. By Najman's classification, only
orders $3$ and $5$ need be considered. Let $P\in E_b(\Q(i))$ 
have prime order $\ell\in\{3,5\}$. If $P$ and $[i]P$ were linearly
independent over $\F_\ell$, then all of $E_b[\ell]$ would be rational
over $\Q(i)$, and the Weil pairing would imply $\mu_\ell\subset\Q(i),$ 
which is impossible. 
Hence,
\[
[i]P=[\lambda]P
\qquad\text{for some}\quad
\lambda\in\F_\ell^\times.
\]
Since $[i]^2=[-1]$, we then have $\lambda^2 = -1$ in $\F_\ell$.  
For $\ell=3$ this has no solution, so $E_b(\Q(i))$ has no nonzero
$3$-torsion.

For $\ell=5$, we have $\lambda=\pm2$, and therefore
\[
x(2P)=x([i]P)=-x(P).
\]
Writing $P=(x,y)$, the duplication formula yields
\[
x(2P)
=\frac{(x^2-b)^2}{4x(x^2+b)}.
\]
Hence, 
\begin{align}\label{5-torsion-quartic-eqn}
5x^4+2bx^2+b^2=0.
\end{align}
Since $x\neq0$, we can define $u:= b / x^2 \in \Q(i)$. Then \eqref{5-torsion-quartic-eqn} is quadratic in $u$ and we obtain $u = -1 \pm 2i$. If $z:=y/x$, then
\[
z^2
=x+\frac{b}{x}
=x(1+u)=
\pm 2ix.
\]
Thus, $x^2 = -z^4 / 4$, and consequently
\[
b
=(-1 \pm 2i)
\left(\frac{z}{1+i}\right)^4.
\]
Therefore, $5$-torsion can occur only if $b \in (-1 \pm 2i) (\Q(i)^\times)^4$.  

Conversely, on $E_{-(1+2i)}$ the point $P=(-1,1+i)$ 
satisfies $[2]P=[i]P$ by the duplication formula. Hence, 
$[2-i]P=O$, and multiplying by $[2+i]$ gives $[5]P=O$. Since
$P\neq O$, it has order $5$. By conjugation,
$E_{-(1-2i)}$ likewise has a rational point of order $5$. Every curve
in either quartic-twist class is $\Q(i)$-isomorphic to one of these
two curves. In these classes $-b$ is not a square, so the
$2$-primary torsion is $\Z/2\Z$. Najman's classification therefore
gives
\[
E_b(\Q(i))_{\mathrm{tors}}\cong\Z/10\Z.
\]

Combining the preceding cases proves the proposition.
\end{proof}

\subsection{Square-split points and support patterns}

We formalize the square-split construction described in the
introduction. The basic point-producing identity is the following.

\begin{lem}\label{lem-factorization-point}
Let $b \in \Q(i)^\times$ and $e,f,R\in\Q(i)$ such that $ef = b$ and $e + f = R^2$. Then  
\[
P=(e,eR)\in E_b(\Q(i))
\qquad\text{and}\qquad
\delta_b(P)=[e].
\]
Conversely, if $P=(x,y)\in E_b(\Q(i))$ with $x\neq0$, then setting
\[
e=x,\qquad f=\frac{b}{x},\qquad R=\frac{y}{x}
\]
gives $ef=b$ and $e+f=R^2$.
\end{lem}

\begin{proof}
This follows immediately from the equation defining $E_b$ and
\eqref{eq-kummer-x}.
\end{proof}

Let $B=u\ell_1\cdots\ell_n$ for some $u \in \Q(i)^\times$, where the $\ell_j$ are nonzero homogeneous binary linear forms. A
\emph{square-split factorization} of $B$ is a factorization $B=ef$ 
of the form
\[
e=c\prod_{j\in S}\ell_j,
\qquad
f=u c^{-1} \prod_{j\notin S}\ell_j,
\qquad
c\in\Q(i)^\times,
\]
for which $e+f=R^2$ with $R$ a nonzero homogeneous form.

\begin{lem}\label{lem-square-split-factor-count}
If $B$ admits a square-split factorization, then $4\mid n.$ 
More precisely,
\[
|S|=n-|S|=\frac n2,
\qquad
\deg R=\frac n4.
\]
\end{lem}

\begin{proof}
The two nonzero homogeneous forms $e$ and $f$ can sum to the
homogeneous form $R^2$ only if they have the same degree. Thus,
\[
|S|=n-|S|=\frac n 2 = 2 \deg R.
\] 
\end{proof}

For a subset $S\subseteq\{1,\dots,n\}$, write
$\mathbf 1_S\in\F_2^n$ for its characteristic vector. We use \emph{weight} to mean Hamming weight for vectors in $\F_2^n$.
A partition is called \textit{balanced} if its two parts have equal cardinality. Complementary subproducts determine the same partition,
while their support vectors differ by the all-ones vector
\[
\mathbf 1=(1,\dots,1).
\]
After specialization to distinct prime factors, $\mathbf 1$
is the support vector of the Kummer class $[b]$ of the torsion point $(0,0)$. For vectors $v_1,\dots,v_r\in\F_2^n$, we write
$\langle v_1,\dots,v_r\rangle$ for their $\F_2$-span.

\begin{lem}\label{lem-square-split-supports}
The following hold.
\begin{enumerate}[label=(\roman*)]
\item For $n=4$, any two distinct balanced partitions are equivalent,
up to relabeling the factors, to
\[
\{1,2\}\sqcup\{3,4\},
\qquad
\{1,3\}\sqcup\{2,4\}.
\]
The support vectors of all balanced partitions span the
$3$-dimensional even-weight subspace of $\F_2^4$. Modulo
$\langle\mathbf 1\rangle$, their images span a $2$-dimensional space.

\item Let $S_1,S_2,S_3\subseteq\{1,\dots,8\}$ have cardinality $4$, and
suppose
\[
C=\langle\one,\mathbf 1_{S_1},\mathbf 1_{S_2},\mathbf 1_{S_3}\rangle
\subseteq\F_2^8
\]
has dimension $4$. If $C$ is self-dual with respect to the standard
dot product, then, after relabeling the eight coordinates, the
coordinates may be indexed by $\F_2^3$ so that
\[
S_j=\{v\in\F_2^3:v_j=0\},
\qquad j=1,2,3.
\]
In particular, $C=\mathrm{RM}(1,3)$.
\end{enumerate}
\end{lem}

\begin{proof}
For (i), a balanced partition of four factors is represented by a
subset of cardinality 2. Any two distinct such partitions
have representatives meeting in one element, and a relabeling gives the desired two partitions. The characteristic vectors of the subsets
of cardinality 2 are precisely the weight $2$ vectors in $\F_2^4$, and these span the even-weight subspace.

For (ii), self-duality implies that the four displayed generators are
pairwise orthogonal. Since each $S_j$ has cardinality $4$, this gives
\[
|S_j\cap S_k|\equiv0\pmod2
\qquad \text{for} \quad j \neq k.
\]
The intersection cannot have size $0$ or size $4$ by the linear independence of the vectors $\mathbf 1, \mathbf 1_{S_j}$, and $\mathbf 1_{S_k}$. Hence,  
\[
|S_j\cap S_k|=2
\qquad \text{for} \quad j\neq k.
\]

For $a=(a_1,a_2,a_3)\in\F_2^3$, let $n_a$ denote the number of
$r\in\{1,\dots,8\}$ such that
\[
\bigl(\mathbf 1_{S_1}(r),\mathbf 1_{S_2}(r),\mathbf 1_{S_3}(r)\bigr)
=(a_1,a_2,a_3).
\] 
The conditions $|S_j|=4$ and $|S_j\cap S_k|=2$ imply
\[
n_{110}=n_{101}=n_{011}=2-n_{111}.
\]
Using $|S_j|=4$ then gives
\[
n_{100}=n_{010}=n_{001}=n_{111}.
\]
Finally, since the eight membership counts sum to $8$, we have $n_{000} = 2 - n_{111}$.  
Thus, every $n_a$ is determined by $n_{111}$. If $n_{111} = 0$, then
\[
\mathbf 1_{S_1}
+
\mathbf 1_{S_2}
+
\mathbf 1_{S_3}
=0,
\]
while if $n_{111} = 2$, then
\[
\mathbf 1_{S_1}
+
\mathbf 1_{S_2}
+
\mathbf 1_{S_3}
=\mathbf 1.
\]
Both contradict linear independence. 

Thus $n_{111}=1$, and so $n_a=1$ for every $a\in\F_2^3$.
Therefore, the map
\[
r\longmapsto
\bigl(\mathbf 1_{S_1}(r),\mathbf 1_{S_2}(r),\mathbf 1_{S_3}(r)\bigr)
\]
is a bijection from $\{1,\dots,8\}$ onto $\F_2^3$. With respect to this labeling,
\[
S_j=\{v\in\F_2^3:v_j=1\},
\qquad j=1,2,3.
\]
Translating all labels by $(1,1,1)$ gives another labeling of the eight coordinates, under which
\[
S_j=\{v\in\F_2^3:v_j=0\},
\qquad j=1,2,3.
\]
The vectors $\mathbf 1_{S_j}$ are then the evaluation vectors of the affine-linear functions $1+x_j$. Hence, 
\[
C=\langle 1,1+x_1,1+x_2,1+x_3\rangle
=\mathrm{RM}(1,3).
\]
\end{proof}

Lemma~\ref{lem-square-split-factor-count} makes four factors the first
possible square-split construction. Part~(i) of
Lemma~\ref{lem-square-split-supports} gives the two support vectors
used in the rank-$4$ construction. At the level of factor supports,
four factors cannot supply a third independent class modulo the
torsion support. Since the number of factors is divisible by $4$, the next possibility is eight factors. For the rank-$6$ construction
we impose the self-duality condition in part~(ii), which determines the
support pattern up to relabeling and yields the three affine coordinate
hyperplanes used. 

The self-duality condition is motivated by the descent. For our
rank-$6$ family we have $s_b=0$, so the odd local conditions of
Proposition~\ref{prop-selmer-linear-system} take the form $Lx=s\nu.$ 
For the three constructed points and $(0,0)$, the corresponding
Kummer classes all have $s=0$. Since these are Selmer classes, their support vectors satisfy $Lx=0$. Hence, if $C$ denotes the span of
these four support vectors, then $C\subseteq\ker L.$ 

Since $C=C^\perp$, we have $\dim C=4$. If
$\operatorname{rank}L=4$, then rank-nullity gives
$\dim\ker L=4$, and therefore $\ker L=C.$ 
Since $L$ is symmetric,
\[
\operatorname{im}L =\operatorname{row}(L) =(\ker L)^\perp =C^\perp =C.
\]
The full system $Lx=s\nu$ can then be read directly from $C$: for
$s=0$ its solutions are exactly the vectors $x\in C$, while for
$s=1$ a solution exists if and only if $\nu\in C$. In the rank-$6$
construction we will show that $\nu\notin C$. Thus,
Proposition~\ref{prop-selmer-linear-system} embeds the Selmer group
into the four-dimensional space
\[
\{0\}\times C\subseteq\F_2\times\F_2^8.
\]
Since the four Kummer classes above are linearly independent, it follows
that
\[
\dim_{\F_2}\Sel_{1+i}(E_b/\Q(i))=4,
\]
which is the sharp Selmer bound needed for the rank-$6$ calculation.

Once the support patterns are fixed, the remaining task is to realize them
by explicit square-split identities. Writing
\[
\ell_j=a_jX+b_jY
\]
and taking a general quadratic form for each $R_k$, the prescribed
identities become a finite system of polynomial equations obtained by
comparing coefficients. The linear forms in
Sections~\ref{sec-rank4} and~\ref{sec-rank6} give explicit solutions
to these systems. Their coefficients are not canonical; the intrinsic features of the
construction are the factor counts and the resulting support patterns.

\subsection{Projective parameters and the associated elliptic surfaces}\label{subsec-projective}
Let $B(X,Y)\in\Q(i)[X,Y]$ 
be homogeneous of degree $4m$, and set
\[
b(T)=B(T,1).
\]
For $\alpha,\beta\in\Q(i)$ with $\beta\neq0$ and
$B(\alpha,\beta)\neq0$, set $t=\alpha/\beta$. By homogeneity,
\[
B(\alpha,\beta)=\beta^{4m}b(t)=(\beta^m)^4b(t).
\]
It follows that the specialized curves are $\Q(i)$-isomorphic:
\begin{align}\label{eq-projective-isomorphism}
E_{B(\alpha,\beta)}
&\xrightarrow{\ \sim\ }
E_{b(t)},\notag\\
(x,y)
&\longmapsto
\left(
\frac{x}{\beta^{2m}},
\frac{y}{\beta^{3m}}
\right).
\end{align}
Hence, the quartic-twist class of $E_{B(\alpha,\beta)}$ depends only on
the projective point $[\alpha:\beta]$.

The homogeneous rank-$4$ coefficient $B_4$ has degree $4$, so its
associated surface is a rational elliptic surface with four fibers of
type $\mathrm{III}$. Its generic fiber has the rational $2$-torsion
point $(0,0)$, and \cite[Corollary~2.1]{oguiso1991mordell} therefore
gives
\[
\operatorname{rank}E(\overline{\Q}(T))\leq4.
\]
The two sections constructed in Section~\ref{sec-rank4}, together with
their $i$-multiples, show that equality holds. 

In contrast, the homogeneous rank-$6$ coefficient $B_6$ has degree
$8$, so its associated surface is an elliptic K3 surface with eight
fibers of type $\mathrm{III}$. 

\subsection{Simultaneous prime values of binary linear forms}

To control the Selmer groups in our families, we need to specialize the linear factors of the twist coefficients to prime elements with prescribed local behavior. For this we use the local-approximation form of Kai's theorem
on linear patterns of prime elements \cite[Proposition~13.2]{kai2023linear}.

\begin{prop}[Linear prime specialization]\label{prop-linear-primes}
Let
\[
\ell_1(X,Y),\dots,\ell_n(X,Y)\in\Z[i][X,Y]
\]
be pairwise nonproportional nonzero homogeneous linear forms. Let
$q\in\Z[i]\setminus\{0\}$ and $(\alpha_0,\beta_0)\in\Z[i]^2$. Assume:
\begin{enumerate}[label=(\roman*)]
\item for every prime ideal $\mathfrak p\mid(q)$ and every $j$, the value
$\ell_j(\alpha_0,\beta_0)$ is a unit at $\mathfrak p$;

\item for every prime ideal $\mathfrak p\nmid(q)$, there is a pair $(\alpha_{\mathfrak p},\beta_{\mathfrak p})
\in(\Z[i]/\mathfrak p)^2$ such that
\[
\ell_1(\alpha_{\mathfrak p},\beta_{\mathfrak p})
\cdots \ell_n(\alpha_{\mathfrak p},\beta_{\mathfrak p})
\not\equiv 0 \pmod{\mathfrak p}.
\]
\end{enumerate}
Then there are infinitely many pairs $(\alpha,\beta)\in\Z[i]^2$ satisfying 
\[
(\alpha,\beta)\equiv(\alpha_0,\beta_0)\pmod q,
\]
for which $\ell_1(\alpha,\beta), \dots, \ell_n(\alpha,\beta)$ are pairwise nonassociate Gaussian primes.
\end{prop}

\begin{proof}
Choose $a\in\Z[i]$ such that none of the $\ell_j$ is proportional to
$X+aY$, and set
\[
\lambda=Y,\qquad \mu=X+aY.
\]
Then we can write
\[
\ell_j=c_j(\lambda-e_j\mu)
\qquad\text{for some}\quad
c_j\in\Q(i)^\times,\; e_j\in\Q(i).
\]
The elements $e_1,\dots,e_n$ are distinct since the $\ell_j$ are
pairwise nonproportional.

Let $S$ be the finite set of nonzero prime ideals $\mathfrak p$ of $\Z[i]$ such that at least one of
\[
\mathfrak p\mid(q),\qquad
v_{\mathfrak p}(e_j)<0\ \text{for some }j,\qquad
v_{\mathfrak p}(c_j)\neq0\ \text{for some }j
\]
holds. Then $e_j\in\Z[i][S^{-1}]$ and $c_j\in\Z[i][S^{-1}]^\times$ for every $j$.  

We now specify the desired local data. For $\mathfrak p\mid(q)$, set
\[
(\lambda_{\mathfrak p},\mu_{\mathfrak p})
=(\beta_0,\alpha_0+a\beta_0) \in \Q(i)_\p^2. 
\]
For $\mathfrak p\in S$ with $\mathfrak p\nmid(q)$, hypothesis~(ii)
provides $(\alpha_{\mathfrak p},\beta_{\mathfrak p})
\in(\Z[i]/\mathfrak p)^2$ such that
\[
\ell_j(\alpha_{\mathfrak p},\beta_{\mathfrak p})
\not\equiv0\pmod{\mathfrak p} \qquad \text{for every}\quad j.
\]
Choose lifts $\widetilde\alpha_{\mathfrak p},
\widetilde\beta_{\mathfrak p}\in\Z[i]_{\mathfrak p}$ and set
\[
(\lambda_{\mathfrak p},\mu_{\mathfrak p})
= (\widetilde\beta_{\mathfrak p},
\widetilde\alpha_{\mathfrak p}
+a\widetilde\beta_{\mathfrak p}) \in \Q(i)_\p^2. 
\]

For each positive integer $k$, apply
\cite[Proposition~13.2]{kai2023linear} with archimedean target
$\tau_k=k\in\C$ and with $\varepsilon>0$ chosen so that
\[
\begin{aligned}
\lambda_k&\equiv\beta_0,\qquad
\mu_k\equiv\alpha_0+a\beta_0
\pmod{\mathfrak p^{v_{\mathfrak p}(q)}}
&&\text{for }\mathfrak p\mid(q),\\
\lambda_k&\equiv \beta_{\mathfrak p},\qquad
\mu_k\equiv
 \alpha_{\mathfrak p}
+a \beta_{\mathfrak p}
\pmod{\mathfrak p}
&&\text{for }\mathfrak p\in S\text{ with } \mathfrak p\nmid(q),
\end{aligned}
\]
and
\[
\left|\frac{\lambda_k}{\mu_k}-k\right|<\frac12.
\]
The proposition yields for each $k \in \Z_{\geq 1}$ a pair $(\lambda_k,\mu_k)\in
(\Z[i][S^{-1}]\setminus\{0\})^2$ with these properties such that the elements
$\lambda_k-e_1\mu_k,\dots,\lambda_k-e_n\mu_k$ are pairwise
nonassociate primes in $\Z[i][S^{-1}]$.

Set $(\alpha_k,\beta_k)
=(\mu_k-a\lambda_k,\lambda_k).$ At the primes dividing $(q)$, the first congruences give
\[
(\alpha_k,\beta_k)\equiv(\alpha_0,\beta_0)
\pmod{\mathfrak p^{v_{\mathfrak p}(q)}},
\]
while at the remaining primes $\mathfrak p\in S$ the second congruences
give
\[
(\alpha_k,\beta_k)\equiv
(\alpha_{\mathfrak p},
\beta_{\mathfrak p})
\pmod{\mathfrak p}.
\]
Thus, $\alpha_k,\beta_k$ are integral at every prime in $S$, and they are
integral away from $S$ since $\lambda_k,\mu_k\in\Z[i][S^{-1}]$. Hence, $(\alpha_k,\beta_k)\in\Z[i]^2$ and 
\[
(\alpha_k,\beta_k)\equiv(\alpha_0,\beta_0)\pmod q.
\]
Moreover, hypothesis~(i) and the choice of the residue classes above imply 
\[
v_{\mathfrak p}(\ell_j(\alpha_k,\beta_k))=0
\qquad \text{for every }\mathfrak p\in S
\]
for every $j \in \{1, \dots, n\}$. 

Finally, since $c_j \in \Z[i][S^{-1}]^\times$, the elements
\[
\ell_j(\alpha_k,\beta_k)
=c_j(\lambda_k-e_j\mu_k), \qquad j \in \{1, \dots, n\}
\] 
are pairwise nonassociate primes in $\Z[i][S^{-1}]$. Since they have valuation zero at every prime in $S$, each
$(\ell_j(\alpha_k,\beta_k))$ is a prime ideal of $\Z[i]$. Thus, the elements $\ell_1(\alpha_k,\beta_k), \dots, \ell_n(\alpha_k, \beta_k)$ are pairwise nonassociate Gaussian primes.

The open disks of radius $1/2$ centered at the positive integers are pairwise disjoint, so the ratios $\lambda_k/\mu_k$ are distinct.
Therefore the corresponding pairs $(\alpha_k,\beta_k)$ are distinct, proving the result.
\end{proof}

\begin{rem}\label{rem-kai-admissibility}
In the applications below, each $\ell_j$ is primitive, so its reduction
modulo every prime ideal is a nonzero linear form. This allows us to check condition~(ii)
by a simple counting argument over the residue field. If $\operatorname{Nm}(\mathfrak p)>n$, then the $n$ lines $\ell_j=0$
cannot cover $(\Z[i]/\mathfrak p)^2$. We therefore choose $q$ so that
$\mathfrak p\mid(q)$ for every prime ideal $\mathfrak p$ with
$\operatorname{Nm}(\mathfrak p)\leq n$. Then condition~(ii) is automatic.
\end{rem}

\subsection{A criterion for genuine definition}

For $j=1728$, a simple valuation criterion rules out descent to
$\Q$.

\begin{lem}\label{lem-nonbasechange}
Let $b\in\Q(i)^\times$. Suppose there is a split Gaussian prime
$\pi$, with associated prime ideal $\mathfrak p=(\pi)$ and its complex conjugate $\overline{\mathfrak p}=(\overline{\pi})$, such that 
\[
v_{\mathfrak p}(b)-v_{\overline{\mathfrak p}}(b)
\not\equiv 0 \pmod 4.
\]
Then $E_b$ is genuinely defined over $\Q(i)$.
\end{lem}
\begin{proof}
If $E_b$ were obtained by base change from $\Q$, then we would have $b = a u^4$ for some $a\in\Q^\times$ and $u\in\Q(i)^\times$. Since $a$ is
rational, $v_{\mathfrak p}(a)=v_{\overline{\mathfrak p}}(a),$ and therefore
\[
v_{\mathfrak p}(b)-v_{\overline{\mathfrak p}}(b)
=4\bigl(v_{\mathfrak p}(u)-v_{\overline{\mathfrak p}}(u)\bigr),
\]
contrary to the hypothesis.
\end{proof}
\section{Rank 4}\label{sec-rank4}

We begin with the one-parameter family. Set
\bal
\lambda_1(T)&=-T,
&\lambda_2(T)&=-i,\\
\lambda_3(T)&=128T-i,
&\lambda_4(T)&=(64-48i)T-i,
\nal
and
\[
b_4(T)=i\lambda_1(T)\lambda_2(T)\lambda_3(T)\lambda_4(T).
\]
The identities
\bal
-4i\lambda_1\lambda_2-\frac14\lambda_3\lambda_4
&=\left(\frac{(32+96i)T+1}{2}\right)^2,\\
-4i\lambda_1\lambda_3-\frac14\lambda_2\lambda_4
&=\left(\frac{(32+32i)T+1}{2}\right)^2
\nal
and Lemma~\ref{lem-factorization-point} give two points on
$E_{b_4(T)}$ over $\Q(i)(T)$, namely
\bal
P_1(T)&=\bigl(4T,\;2T((32+96i)T+1)\bigr),\\
P_2(T)&=\bigl(4T(128iT+1),\;2T(128iT+1)((32+32i)T+1)\bigr).
\nal

Homogenize the four factors by setting
\[
\ell_j(X,Y)=Y\lambda_j(X/Y) \qquad \text{for }j \in \{1,2,3,4\},
\]
and set
\[
B_4(X,Y)=i\ell_1\ell_2\ell_3\ell_4=Y^4b_4(X/Y).
\]
Thus, $E_{B_4(X,Y)}$ and $E_{b_4(X/Y)}$ are related by \eqref{eq-projective-isomorphism}. Homogenizing the two square-split identities gives corresponding points $\widetilde P_1,\widetilde P_2$ on
$E_{B_4(X,Y)}$. Since
\[
XY=-i\ell_1\ell_2,
\qquad
X(128iX+Y)=-i\ell_1\ell_3,
\]
their Kummer classes are
\[
\delta_{B_4}(\widetilde P_1)=[i\ell_1\ell_2],
\qquad
\delta_{B_4}(\widetilde P_2)=[i\ell_1\ell_3].
\]
Under \eqref{eq-projective-isomorphism}, these points specialize to
$P_1(t)$ and $P_2(t)$ when $t=\alpha/\beta$.

The following elementary lemma largely determines the rank-$4$ descent
from these two Kummer classes.

\begin{lem}\label{lem-rank4-rigid}
Let $M\in\operatorname{Mat}_{4\times4}(\F_2)$ be symmetric with
$M\one=\one$. Suppose
\[
M(1,1,0,0)^T=M(1,0,1,0)^T=\one.
\]
Then $\operatorname{rank}M=2$ and
\[
\#\{(s,x)\in\F_2\times\F_2^4:Mx=s\one\}=8.
\]
\end{lem}

\begin{proof}
Write the columns of $M$ as $c_1,\dots,c_4$. The hypotheses give
\[
c_1+c_2=c_1+c_3=\one,
\]
so $c_2=c_3$. Since $M\one=\one$, we also have
\[
c_1+c_2+c_3+c_4=\one,
\]
and hence $c_4=c_2$. Thus
\[
c_2=c_3=c_4=:c,
\qquad
c_1=\one+c.
\]
By symmetry, the last three coordinates of $c$ are equal, while its
first coordinate is their complement. Hence $c\notin\{0,\one\}$, so
$c$ and $\one+c$ are linearly independent. Therefore
$\operatorname{rank}M=2$. Moreover, $\one=c_1+c_2$ lies in the image
of $M$. Thus, for each $s\in\F_2$, the equation $Mx=s\one$ has
$2^{4-2}=4$ solutions, giving eight pairs $(s,x)$ in total.
\end{proof}

\begin{rem}
When $M=L+I_4$, as in the rank-$4$ descent below, the proof of
Lemma~\ref{lem-rank4-rigid} has a graph-theoretic interpretation. The
two possibilities for $M$ correspond exactly to the quadratic-residue
graph being the star $K_{1,3}$, with the $\pi_1$ vertex at its center, or its
complement. 
\end{rem}

For the local conditions below, fix
\[
\alpha_0=1+2i,
\qquad
\beta_0=42-13i.
\]
At this specialization, the four values
$\ell_1(\alpha_0,\beta_0),\dots,\ell_4(\alpha_0,\beta_0)$ are
\bal
-1-2i,
\quad -13-42i,
\quad 115+214i,
\quad 147+38i,
\nal
with respective norms
\[
5,
\quad 1933,
\quad 59021,
\quad 23053.
\]
These norms are rational primes, so the four values are Gaussian primes.
Moreover, each is congruent to $3+2i$ modulo $4$, and therefore primary. In the notation of Subsection~\ref{subsection-odd-local-descent}, this also shows that $\nu=\one$.

\begin{prop}\label{prop-rank4-family}
Suppose $(\alpha,\beta)\in\Z[i]^2$ satisfies
\[
(\alpha,\beta)\equiv(\alpha_0,\beta_0)\pmod{28},
\]
and set
\[
\pi_j=\ell_j(\alpha,\beta),
\qquad
\mathfrak p_j=(\pi_j)
\qquad
\text{for }j \in \{1,2,3,4\}.
\]
Further suppose that $\pi_1, \pi_2, \pi_3, \pi_4$ are Gaussian primes, and let
\[
b=B_4(\alpha,\beta).
\]
Then $E_b$ is genuinely defined over $\Q(i)$, has Mordell--Weil group
\[
E_b(\Q(i))\cong\Z^4\oplus\Z/2\Z,
\]
and satisfies
\[
\Sh(E_b/\Q(i))[2^\infty]=0.
\]
\end{prop}

\begin{proof}
Modulo $7$, the Gaussian primes $\pi_1, \dots, \pi_4$ are congruent to
\[
6+5i,
\quad
1,
\quad
3+4i,
\quad
3i,
\]
respectively. These residues lie in distinct orbits under multiplication
by $\Z[i]^\times=\{\pm1,\pm i\}$, so the $\pi_j$ are pairwise
nonassociate. The congruence modulo $4$ implies that they are primary
and that $\nu=\one$. Since
\[
b=i\pi_1\pi_2\pi_3\pi_4,
\]
we have $s_b=1$ in the notation of Proposition~\ref{prop-selmer-linear-system}. The proposition therefore
shows that every Selmer class is uniquely represented by a pair
$(s,x)\in\F_2\times\F_2^4$ satisfying
\begin{align}\label{eq-Laplacian-in-proof}
(L+I_4)x=s\one.
\end{align}
Since $L$ is symmetric and $L\one=0$, the matrix $M:=L+I_4$ is symmetric and satisfies $M\one=\one$. The specializations of $\widetilde P_1$ and $\widetilde P_2$ correspond to
\[
(s,x)=\bigl(1,(1,1,0,0)\bigr),
\; 
\bigl(1,(1,0,1,0)\bigr).
\]
Lemma~\ref{lem-rank4-rigid} shows that the equation \eqref{eq-Laplacian-in-proof} 
has exactly eight solutions $(s,x)\in\F_2\times\F_2^4$. Hence, 
\[
\dim_{\F_2}\Sel_{1+i}(E_b/\Q(i))\leq3.
\]

On the other hand, the three squareclasses $[b], [i \pi_1 \pi_2], [i \pi_1 \pi_3]$ are linearly independent. Indeed, reducing the
valuations at $\mathfrak p_4,\mathfrak p_2,\mathfrak p_3$ modulo $2$ successively forces the coefficients in any $\F_2$-linear relation among them to vanish. They are the Kummer classes of
$(0,0)$ and the specializations of $\widetilde P_1,\widetilde P_2$, respectively. Therefore, 
Lemma~\ref{lem-sharp-descent} with $m=2$ yields
\[
\operatorname{rk}_{\Z}E_b(\Q(i))=4
\qquad \text{and}\qquad 
\Sh(E_b/\Q(i))[2^\infty]=0.
\]

Since $b=i\pi_1\pi_2\pi_3\pi_4$ has odd valuation at four distinct odd
primes, $b \notin -(\Q(i)^\times)^2$ and $b \notin  (-1 \pm 2i) (\Q(i)^\times)^4$. Hence, Proposition \ref{prop-torsion} gives us
\[
E_b(\Q(i))_{\mathrm{tors}}\cong\Z/2\Z.
\]

Finally, the four unit multiples of $6+2i$ modulo $7$ are
\[
\{6+2i,\;5+6i,\;1+5i,\;2+i\},
\]
none of which is congruent to one of the four prime values above. Thus,
$\overline{\pi_1}$ is nonassociate to every $\pi_j$, and so 
\[
v_{\mathfrak p_1}(b)=1,
\qquad
v_{\overline{\mathfrak p}_1}(b)=0.
\]
Therefore, $E_b$ is genuinely defined over $\Q(i)$ by Lemma~\ref{lem-nonbasechange}. 
\end{proof}

\begin{thm}\label{thm-rank4-infinite}
There are infinitely many pairwise nonisomorphic elliptic curves
$E/\Q(i)$ genuinely defined over $\Q(i)$ with $j(E)=1728$ such that
\[
E(\Q(i))\cong\Z^4\oplus\Z/2\Z
\]
and
\[
\Sh(E/\Q(i))[2^\infty]=0.
\]
\end{thm}

\begin{proof}
Apply Proposition~\ref{prop-linear-primes} to the four forms
$\ell_1,\dots,\ell_4$ with
\[
q=28,
\qquad
(\alpha,\beta)\equiv(\alpha_0,\beta_0)\pmod{28}.
\]
The four forms are pairwise nonproportional and primitive, and the
four values $\ell_j(\alpha_0,\beta_0)$ are units at every prime ideal
dividing $(28)$. Since every prime ideal not dividing $(28)$ has norm
at least $5>4$, condition~(ii) holds by Remark~\ref{rem-kai-admissibility}. Proposition~\ref{prop-linear-primes} therefore gives infinitely many pairs $(\alpha,\beta)$ in this congruence class such that the four
values $\ell_j(\alpha,\beta)$ are pairwise nonassociate Gaussian primes. Hence, 
each of these pairs satisfies the hypotheses of
Proposition~\ref{prop-rank4-family}.

Since $\ell_2(\alpha,\beta)=-i\beta$ is prime, we have $\beta\neq 0$, so
$t=\alpha/\beta$ is defined; moreover,
$B_4(\alpha,\beta)=\beta^4b_4(t)$, so 
$E_{B_4(\alpha,\beta)}\cong E_{b_4(t)}$ over $\Q(i)$. Suppose two such pairs $(\alpha,\beta)$ and $(\alpha',\beta')$ give the
same value of $t$. Then
\[
(\alpha',\beta')=\lambda(\alpha,\beta)
\]
for some $\lambda\in\Q(i)^\times$, and so
$\ell_j(\alpha',\beta')=\lambda\ell_j(\alpha,\beta)$ for every $j$.
Since both sets of values consist of pairwise nonassociate Gaussian
primes, $\lambda$ can have no nonzero prime valuation and must therefore
be a unit. Thus, each value of $t$ arises from at most four such pairs,
so Proposition~\ref{prop-linear-primes} yields infinitely many values of
$t$.

If two such specializations determine the same quartic-twist class, then
their twist coefficients have the same valuations modulo $4$ at every
prime ideal. Since each coefficient is the product of four pairwise
nonassociate odd Gaussian primes, the two specializations therefore
involve the same four odd prime ideals. After permuting these ideals,
each value $\ell_j(\alpha,\beta)$ is determined up to multiplication by
a unit, so there are only finitely many possible ordered quadruples of
prime values. Since $\ell_1$ and $\ell_2$ are nonproportional, the values
of these two forms determine $(\alpha,\beta)$. Hence, each quartic-twist
class arises from only finitely many such pairs, and the infinitely many
specializations above yield infinitely many pairwise nonisomorphic
curves over $\Q(i)$.
\end{proof}

\section{Rank 6}\label{sec-rank6}

We now realize the eight-factor support pattern of
Lemma~\ref{lem-square-split-supports}(ii) by an explicit one-parameter
family. Set
\bal
\lambda_0(T)&=T+1,
&\lambda_1(T)&=3T+1,\\
\lambda_2(T)&=-6iT-2-5i,
&\lambda_3(T)&=(8+10i)T+2+3i,\\
\lambda_4(T)&=(-2+2i)T+i,
&\lambda_5(T)&=(-6-6i)T+i,\\
\lambda_6(T)&=2iT+i,
&\lambda_7(T)&=(24-30i)T+16-11i,
\nal
and
\[
b_6(T)=\prod_{j=0}^7\lambda_j(T).
\]
Also define
\bal
R_1(T)&=42T^2+(42-4i)T+\frac{19}{2},\\
R_2(T)&=\frac{264-108i}{5}T^2+\frac{138-56i}{5}T+\frac{9-3i}{5},\\
R_3(T)&=(60-24i)T^2+(40-18i)T+7-5i.
\nal
Set $\eta=(2+ i)/5$. Then the following square-split identities hold:
\bal
4\lambda_0\lambda_1\lambda_4\lambda_5
+4^{-1}\lambda_2\lambda_3\lambda_6\lambda_7
&=R_1^2,\\
\eta^2\lambda_0\lambda_2\lambda_4\lambda_7
+\eta^{-2}\lambda_1\lambda_3\lambda_5\lambda_6
&=R_2^2,\\
-\lambda_0\lambda_2\lambda_5\lambda_6
-\lambda_1\lambda_3\lambda_4\lambda_7
&=R_3^2.
\nal
Thus, with
\bal
e_1(T)&=4\lambda_0\lambda_1\lambda_4\lambda_5,\\
e_2(T)&=\eta^2\lambda_0\lambda_2\lambda_4\lambda_7,\\
e_3(T)&=-\lambda_0\lambda_2\lambda_5\lambda_6,
\nal
the curve $E_{b_6(T)}$ over $\Q(i)(T)$ has the three points
\[
P_j(T)=\bigl(e_j(T),e_j(T)R_j(T)\bigr),
\qquad j=1,2,3.
\]

Homogenize the eight factors by setting
\[
\ell_j(X,Y)=Y\lambda_j(X/Y),
\qquad 0\leq j\leq7,
\]
and set
\[
B_6(X,Y)=\prod_{j=0}^7\ell_j(X,Y)
= (Y^2)^4 b_6(X/Y).
\]
Thus $E_{B_6(X,Y)}$ and $E_{b_6(X/Y)}$ are related by
\eqref{eq-projective-isomorphism}. Homogenizing the three square-split
identities gives corresponding points
$\widetilde P_1,\widetilde P_2,\widetilde P_3$ on $E_{B_6(X,Y)}$.
Since $4$, $\eta^2$, and $-1$ are squares in $\Q(i)$, their Kummer
classes are
\bal
\delta_{B_6}(\widetilde P_1)
&=[\ell_0\ell_1\ell_4\ell_5],\\
\delta_{B_6}(\widetilde P_2)
&=[\ell_0\ell_2\ell_4\ell_7],\\
\delta_{B_6}(\widetilde P_3)
&=[\ell_0\ell_2\ell_5\ell_6].
\nal
Under \eqref{eq-projective-isomorphism}, these points specialize to
$P_j(t)$ when $t=\alpha/\beta$.

We identify $0,1,\ldots,7$ with
\begin{align}\label{eq-F2^3-labels}
(0,0,0),(0,1,1),(1,0,0),(1,1,1),(0,0,1),(0,1,0),(1,1,0),(1,0,1)
\in\F_2^3,
\end{align}
respectively. Then the supports of the three Kummer classes above are precisely
\[
H_j=\{v\in\F_2^3:v_j=0\},
\qquad j=1,2,3.
\]
Together with the all-ones support of the Kummer class of $(0,0)$,
they span
\begin{equation}\label{eq-rm-code}
C:=\langle\one_{H_1},\one_{H_2},\one_{H_3},\one\rangle
=\mathrm{RM}(1,3)\subset\F_2^8.
\end{equation}
Equivalently, $C$ is the space of evaluation vectors of affine-linear
functions on $\F_2^3$.
 
Set
\[
q_6=(1+i)^7\cdot3\cdot5\cdot13\cdot17\cdot41,
\]
and fix
\begin{equation}\label{eq-rank6-base-specialization}
\alpha_0=-5-5i,
\qquad
\beta_0=4+11i.
\end{equation}
At this specialization, the eight Gaussian integers
$\ell_0(\alpha_0,\beta_0), \dots, \ell_7(\alpha_0, \beta_0)$ are
\begin{equation}\label{eq-rank6-base-primes}
\begin{gathered}
-1+6i,
\quad -11-4i,
\quad 17-12i,
\quad -15-56i,\\
9+4i,
\quad -11+64i,
\quad -1-6i,
\quad -85+162i.
\end{gathered}
\end{equation}
Their norms are the rational primes
\[
37,
\quad137,
\quad433,
\quad3361,
\quad97,
\quad4217,
\quad37,
\quad33469,
\]
so they are each a Gaussian prime. Moreover, they are pairwise nonassociate and primary, and in the notation of
Subsection~\ref{subsection-odd-local-descent},
\begin{equation}\label{eq-rank6-nu}
\nu=(1,0,0,0,0,0,1,1)^T.
\end{equation}
Let $G_0$ denote their quadratic-residue graph. We now show that the
congruence condition modulo $q_6$ determines this graph up to complementation.

\begin{lem}\label{lem-rank6-residue-graphs}
Suppose $(\alpha,\beta)\in\Z[i]^2$ satisfies
\[
(\alpha,\beta)\equiv(\alpha_0,\beta_0)\pmod{q_6},
\]
and that the eight Gaussian integers
\[
\pi_j=\ell_j(\alpha,\beta),
\qquad j=0,\ldots,7,
\]
are pairwise nonassociate Gaussian primes. Then the $\pi_j$ are primary,
and their quadratic-residue graph is either $G_0$ or its complement.
\end{lem}

\begin{proof}
For every $j\in\{0,\ldots,7\}$, let $A_j,B_j\in\Z[i]$ be such that
\[
\ell_j(X,Y)=A_jX+B_jY.
\]
Also define
\[
D_{j,k}=A_jB_k-A_kB_j
\qquad \text{for all }0\leq j<k\leq 7,
\]
and set 
\[
\Delta = \prod_{j=0}^7A_j \prod_{0\leq j<k\leq7}D_{j,k}. 
\]
A direct norm calculation gives
\bal
\prod_{0\leq j<k\leq7}\operatorname{Nm}(D_{j,k})
&=2^{58}3^{12}5^{28}13^6 17^6,\\
\prod_{j=0}^7\operatorname{Nm}(A_j)
&=2^{14}3^8 41^2,
\nal
and so
\[
\operatorname{Nm}(\Delta)
=2^{72}3^{20}5^{28}13^6 17^6 41^2.
\]
Thus, every prime ideal dividing some $A_j$ or $D_{j,k}$ lies above one of the rational primes
\[
2,3,5,13,17,41,
\]
and therefore divides $(q_6)$. Since $\pi_j=\ell_j(\alpha,\beta)$, we have
$\pi_j\equiv\ell_j(\alpha_0,\beta_0)\pmod{q_6}$. Hence, the $\pi_j$ are
units at every prime ideal dividing $(q_6)$, and the congruence modulo
$(1+i)^7$ implies that they are primary.

In order to apply quadratic reciprocity, set $\beta'=i\beta$. Since
$\beta\equiv\beta_0\pmod{q_6}$ and $i\beta_0=-11+4i$ is primary and a
unit at every prime ideal dividing $(q_6)$, we have
$\beta'\equiv-11+4i\pmod{q_6}$. Hence, $\beta'$ is primary and is a unit
at every prime ideal dividing $(q_6)$.

We claim that $\beta'$ is coprime to every $\pi_j$. Indeed, suppose
that a prime ideal $\mathfrak p\nmid(q_6)$ divides both $\beta'$ and
$\pi_j$. Then $\mathfrak p\mid\beta$, while $A_j$ is a unit at
$\mathfrak p$ since every prime divisor of $A_j$ divides $(q_6)$. From
\[
\pi_j=A_j\alpha+B_j\beta
\]
it follows that $\mathfrak p\mid\alpha$, and hence
$\mathfrak p\mid\pi_k$ for every $k$, contradicting that the $\pi_k$
are pairwise nonassociate. Since $\beta'$ is a unit at every prime
dividing $(q_6)$, this proves the claim. The same argument shows that
$\alpha$ and $\beta'$ are coprime.

For $0\leq j<k\leq7$, we have
\[
A_j\pi_k-A_k\pi_j=D_{j,k}\beta.
\]
Reducing modulo $\mathfrak p_j=(\pi_j)$ then gives
\[
\pi_k\equiv D_{j,k}A_j^{-1}\beta\pmod{\mathfrak p_j}.
\]
Recall that $\beta=-i\beta'$. We may also replace $A_j^{-1}$ by $A_j$,
since the quadratic residue symbol depends only on the squareclass of
its numerator. Then 
quadratic reciprocity, together with
\[
\pi_j\equiv A_j\alpha\pmod{\beta'},
\]
yields 
\bal
\left(\frac{\pi_k}{\pi_j}\right)
&=\left(\frac{-iD_{j,k}A_j}{\pi_j}\right)
  \left(\frac{\beta'}{\pi_j}\right)\\
&=\left(\frac{-iD_{j,k}A_j}{\pi_j}\right)
  \left(\frac{\pi_j}{\beta'}\right)\\
&=\left(\frac{-iD_{j,k}A_j}{\pi_j}\right)
  \left(\frac{A_j}{\beta'}\right)
  \left(\frac{\alpha}{\beta'}\right).
\nal
The first two factors are determined by the congruence class modulo
$q_6$. Hence, for each $0\leq j<k\leq7$, there is a constant
$c_{j,k}\in\{\pm1\}$ such that
\[
\left(\frac{\pi_k}{\pi_j}\right)
=c_{j,k}\kappa,
\qquad \text{where}\quad
\kappa:=\left(\frac{\alpha}{\beta'}\right)\in\{\pm1\}.
\]
Thus, all $28$ pairwise Legendre symbols are determined by
$(\alpha,\beta)$ modulo $q_6$, up to the single sign $\kappa$. Since
$(\alpha,\beta)\equiv(\alpha_0,\beta_0)\pmod{q_6}$, the constants
$c_{j,k}$ agree with those for $(\alpha_0,\beta_0)$. Hence, the
quadratic-residue graph of $\pi_0,\ldots,\pi_7$ is $G_0$ when $\kappa$
has the same value as for $(\alpha_0,\beta_0)$, and is the complement
of $G_0$ when its sign is reversed.
\end{proof}

A direct calculation of the residue symbols at
\eqref{eq-rank6-base-specialization} gives the edge set
\[
01,02,04,06,16,17,25,26,34,35,36,37,45,47,
\]
where $jk$ denotes the edge $\{j,k\}$. Write $v=(a,b,c)$ and $w=(x,y,z)$ for elements of $\F_2^3$, viewed as
vertices of $G_0$ under the labeling \eqref{eq-F2^3-labels}. A direct
check from the edge set above shows that the $(v,w)$-entry of the
Laplacian $L_0$ is
\[
(L_0)_{v,w}=ax+by+a+x+c+z.
\]

For an eight-vertex graph, taking the complement adds $1$ to every
entry of its Laplacian over $\F_2$. Thus, the two possible Laplacians of
the quadratic-residue graph of $\pi_0,\ldots,\pi_7$ are
\[
(L_\epsilon)_{v,w} =ax+by+a+x+c+z+\epsilon,
\qquad
\epsilon\in\F_2.
\]

\begin{lem}\label{lem-rank6-kernel}
For each $\epsilon\in\F_2$,
\[
\operatorname{row}(L_\epsilon)=\ker(L_\epsilon)=C.
\]
Moreover, $\nu\notin C$.
\end{lem}

\begin{proof}
For $v=(a,b,c)\in\F_2^3$, let $R_v$ denote the row of $L_\epsilon$
indexed by $v$. As $w=(x,y,z)$ varies over $\F_2^3$, the entries of
$R_v$ are given by the affine-linear function
\[
(x,y,z)\longmapsto (a+1)x+by+z+(a+c+\epsilon).
\]
Hence $R_v\in C$ for every $v$, so $\operatorname{row}(L_\epsilon)\subseteq C.$ 

Conversely, identifying affine-linear functions with their evaluation
vectors, we have
\bal
R_{v+(1,0,0)}+R_v &=x+1,\\
R_{v+(0,1,0)}+R_v &=y,\\
R_{v+(0,0,1)}+R_v &=1,
\nal
and $R_{(0,0,0)}=x+z+\epsilon.$ 
Thus, $\operatorname{row}(L_\epsilon)$ contains $1,x,y,z$, and so $C\subseteq\operatorname{row}(L_\epsilon).$

The evaluation vectors of $1,x,y,z$ are pairwise orthogonal and each
has even weight, so $C\subseteq C^\perp$. Since $\dim C=4$, it follows
that $C=C^\perp$. Since $L_\epsilon$ is symmetric, 
\[
\ker L_\epsilon
=\operatorname{row}(L_\epsilon)^\perp
=C.
\]
Finally, the vector $\nu$ (given by \eqref{eq-rank6-nu}) has odd weight $3$, while every
vector in $C$ has even weight, so $\nu\notin C$.
\end{proof}

\begin{prop}\label{prop-rank6-family}
Suppose $(\alpha,\beta)\in\Z[i]^2$ satisfies
\[
(\alpha,\beta)\equiv(\alpha_0,\beta_0)\pmod{q_6},
\]
and that the eight Gaussian integers
\[
\pi_j=\ell_j(\alpha,\beta),
\qquad j \in \{0,\dots, 7\}, 
\]
are Gaussian primes. Set $\p_j = (\pi_j)$ for $j \in \{0, \dots, 7\}$ and let 
\[
b=B_6(\alpha,\beta)=\prod_{j=0}^7\pi_j.
\]
Then $E_b$ is genuinely defined over $\Q(i)$, has Mordell--Weil group
\[
E_b(\Q(i))\cong\Z^6\oplus\Z/2\Z,
\]
and satisfies
\[
\Sh(E_b/\Q(i))[2^\infty]=0.
\]
\end{prop}

\begin{proof}
Modulo $13$, the eight prime values are congruent to
\[
12+6i,
\quad 2+9i,
\quad 4+i,
\quad 11+9i,
\quad 9+4i,
\quad 2+12i,
\quad 12+7i,
\quad 6+6i.
\]
These residues lie in distinct orbits under multiplication by
$\Z[i]^\times$, so the $\pi_j$ are pairwise nonassociate. The
congruence modulo $(1+i)^7$ implies that the associated vector $\nu$ is the
one in \eqref{eq-rank6-nu}, while
Lemma~\ref{lem-rank6-residue-graphs} shows that their
quadratic-residue Laplacian is $L_\epsilon$ for some
$\epsilon\in\F_2$.

From the definition of $b$ we have $s_b = 0$. Proposition~\ref{prop-selmer-linear-system} and
Lemma~\ref{lem-rank6-kernel} therefore give
\[
\Sel_{1+i}(E_b/\Q(i))
\hookrightarrow
\{0\}\times C,
\]
which implies 
\[
\dim_{\F_2}\Sel_{1+i}(E_b/\Q(i))\leq4.
\]
The three homogeneous points constructed above, together with $(0,0)$,
have the four Kummer classes whose support vectors span $C$ in
\eqref{eq-rm-code}. These classes are therefore independent.
Thus, by Lemma~\ref{lem-sharp-descent} with $m=3$, we have
\[
\operatorname{rk}_{\Z}E_b(\Q(i))=6
\qquad \text{and}\qquad
\Sh(E_b/\Q(i))[2^\infty]=0.
\]
Moreover, $E_b$ has the asserted $\Q(i)$-torsion by Proposition~\ref{prop-torsion}, since $b \notin -(\Q(i)^\times)^2$ and $b \notin (-1 \pm 2i)(\Q(i)^\times)^4$ as $b$ has odd valuation at eight distinct odd primes. 

To show that $E_b$ is genuinely defined over $\Q(i)$, note that the four unit multiples of $4-i$ modulo $13$ are
\[
\{4+12i,\;1+4i,\;9+i,\;12+9i\},
\]
none of which is congruent to one of the eight prime values above. Thus,
$\overline{\pi_2}$ is nonassociate to every $\pi_j$, and so
\[
v_{\mathfrak p_2}(b)=1,
\qquad
v_{\overline{\mathfrak p}_2}(b)=0.
\]
Lemma~\ref{lem-nonbasechange} therefore shows that $E_b$ is genuinely
defined over $\Q(i)$.
\end{proof}

\begin{thm}\label{thm-rank6-infinite}
There are infinitely many pairwise nonisomorphic elliptic curves
$E/\Q(i)$ genuinely defined over $\Q(i)$ with $j(E)=1728$ such that
\[
E(\Q(i))\cong\Z^6\oplus\Z/2\Z
\]
and
\[
\Sh(E/\Q(i))[2^\infty]=0.
\]
\end{thm}

\begin{proof}
Apply Proposition~\ref{prop-linear-primes} to the eight forms
$\ell_0,\dots,\ell_7$ with
\[
q=q_6,
\qquad
(\alpha,\beta)\equiv(\alpha_0,\beta_0)\pmod{q_6}.
\]
The eight forms are pairwise nonproportional and primitive, and the
Gaussian primes in \eqref{eq-rank6-base-primes} are units at every
prime ideal dividing $(q_6)$. Since every prime ideal not dividing
$(q_6)$ has norm greater than $8$, condition~(ii) holds by
Remark~\ref{rem-kai-admissibility}. Proposition~\ref{prop-linear-primes}
therefore gives infinitely many pairs $(\alpha,\beta)$ for which the
eight forms take pairwise nonassociate Gaussian prime values. Hence,
each of these pairs satisfies the hypotheses of
Proposition~\ref{prop-rank6-family}.

The congruence $\beta\equiv\beta_0\pmod{q_6}$ implies $\beta\neq0$, so
$t=\alpha/\beta$ is defined. Moreover, $E_{B_6(\alpha,\beta)}\cong E_{b_6(t)}$ over $\Q(i)$ since
\[
B_6(\alpha,\beta)=\beta^8b_6(t),
\] 
Thus, Proposition~\ref{prop-rank6-family} shows that the
corresponding curves $E_{b_6(t)}$ are genuinely defined over $\Q(i)$
and satisfy
\[
E_{b_6(t)}(\Q(i))\cong\Z^6\oplus\Z/2\Z,
\qquad
\Sh(E_{b_6(t)}/\Q(i))[2^\infty]=0.
\]

As in the proof of Theorem~\ref{thm-rank4-infinite}, the pairs supplied
by Proposition~\ref{prop-linear-primes} yield infinitely many values of
$t$. The same finite-to-one argument, using the nonproportional forms
$\ell_0$ and $\ell_1$, shows that each quartic-twist class occurs only
finitely often. Therefore, infinitely many of the resulting curves are
pairwise nonisomorphic over $\Q(i)$.
\end{proof}
 
\begingroup
\renewcommand*{\bibfont}{\footnotesize}
\emergencystretch=2em
\printbibliography
\endgroup
\end{document}